\documentclass[11pt]{article}
\usepackage{amsmath,amssymb,amsthm,amscd,enumerate}
\usepackage{xcolor}
\usepackage[all]{xy}
\newtheorem{theorem}{Theorem}[section]
\newtheorem{lemma}[theorem]{Lemma}
\newtheorem{proposition}[theorem]{Proposition}
\newtheorem{corollary}[theorem]{Corollary}
\newtheorem{example}[theorem]{Example}

\theoremstyle{plain}

\theoremstyle{definition}
\newtheorem{definition}[theorem]{Definition}
\newtheorem{remark}[theorem]{Remark}

\numberwithin{equation}{section}

\renewcommand{\theenumi}{(\roman{enumi})}
\renewcommand{\labelenumi}{\textup{(\theenumi)}}

\title{Total extension groups for unital Kirchberg algebras
}
\author{Kengo Matsumoto \\
Department of Mathematics \\
Joetsu University of Education \\
Joetsu, Niigata 943-8512, Japan
\and
Taro Sogabe \\
Department of Mathematics \\
Kyoto University \\
Kyoto,  606-8502, Japan}

\begin{document}


\maketitle

\date{}

\def\det{{{\operatorname{det}}}}

\begin{abstract}
We introduce a hierarchy for unital Kirchberg algebras with finitely generated K-groups 
by which the first and second homotopy groups of the automorphism groups 
serve as a complete invariant of classification.
We also introduce an invariant called the total extension group 
which is the direct sum of the strong and weak extension groups. 
In the case of unital Kirchberg algebras  with finitely generated K-groups,
the total extension group gives a complete invariant and provides a useful tool to classify the Cuntz--Krieger algebras.
\end{abstract}

{\it Mathematics Subject Classification}:
Primary 46L80; Secondary 19K33, 19K35.

{\it Keywords and phrases}: Homotopy group, K-group, Ext-group, KK-theory,  
$C^*$-algebra,  Kirchberg algebra, duality, automorphism group.



\newcommand{\Ker}{\operatorname{Ker}}
\newcommand{\sgn}{\operatorname{sgn}}
\newcommand{\Ad}{\operatorname{Ad}}
\newcommand{\ad}{\operatorname{ad}}
\newcommand{\orb}{\operatorname{orb}}
\newcommand{\rank}{\operatorname{rank}}

\def\Re{{\operatorname{Re}}}
\def\det{{{\operatorname{det}}}}
\newcommand{\K}{\operatorname{K}}
\newcommand{\bbK}{\mathbb{K}}
\newcommand{\N}{\mathbb{N}}
\newcommand{\bbC}{\mathbb{C}}
\newcommand{\R}{\mathbb{R}}
\newcommand{\Rp}{{\mathbb{R}}^*_+}
\newcommand{\T}{\mathcal{T}}

\newcommand{\sqK}{\operatorname{K}\!\operatorname{K}}

\newcommand{\Z}{\mathbb{Z}}
\newcommand{\Zp}{{\mathbb{Z}}_+}
\def\AF{{{\operatorname{AF}}}}

\def\TorZ{{{\operatorname{Tor}}^\Z_1}}
\def\Ext{{{\operatorname{Ext}}}}
\def\Exts{\operatorname{Ext}_{\operatorname{s}}}
\def\Extw{\operatorname{Ext}_{\operatorname{w}}}
\def\Extt{\operatorname{Ext}_{\operatorname{t}}}
\def\Ext{{{\operatorname{Ext}}}}
\def\Free{{{\operatorname{Free}}}}

\def\Ks{\operatorname{K}^{\operatorname{s}}}
\def\Kw{\operatorname{K}^{\operatorname{w}}}

\def\OA{{{\mathcal{O}}_A}}
\def\ON{{{\mathcal{O}}_N}}
\def\OAT{{{\mathcal{O}}_{A^t}}}
\def\OAI{{{\mathcal{O}}_{A^\infty}}}
\def\OAIn{{{\mathcal{O}}_{A^\infty_n}}}
\def\OAInone{{{\mathcal{O}}_{A^\infty_{n+1}}}}

\def\OAIN{{{\mathcal{O}}_{A^\infty_N}}}
\def\OATI{{{\mathcal{O}}_{A^{t \infty}}}}

\def\OSA{{{\mathcal{O}}_{S_A}}}

\def\TA{{{\mathcal{T}}_A}}
\def\TAn{{{\mathcal{T}}_{A_n}}}
\def\TAnone{{{\mathcal{T}}_{A_{n+1}}}}
\def\TAN{{{\mathcal{T}}_{A_N}}}

\def\TN{{{\mathcal{T}}_N}}

\def\TAT{{{\mathcal{T}}_{A^t}}}

\def\TB{{{\mathcal{T}}_B}}
\def\TBT{{{\mathcal{T}}_{B^t}}}

\def\A{{\mathcal{A}}}
\def\B{{\mathcal{B}}}
\def\C{{\mathcal{C}}}
\def\O{{\mathcal{O}}}
\def\OaA{{{\mathcal{O}}^a_A}}
\def\OB{{{\mathcal{O}}_B}}
\def\OTA{{{\mathcal{O}}_{\tilde{A}}}}
\def\F{{\mathcal{F}}}
\def\G{{\mathcal{G}}}
\def\FA{{{\mathcal{F}}_A}}
\def\PA{{{\mathcal{P}}_A}}
\def\PI{{{\mathcal{P}}_\infty}}
\def\OI{{{\mathcal{O}}_\infty}}

\def\whatA{{\widehat{\A}}}

\def\calI{\mathcal{I}}

\def\bbC{{\mathbb{C}}}

 \def\U{{\mathcal{U}}}
\def\OF{{{\mathcal{O}}_F}}
\def\DF{{{\mathcal{D}}_F}}
\def\FB{{{\mathcal{F}}_B}}
\def\DA{{{\mathcal{D}}_A}}
\def\DB{{{\mathcal{D}}_B}}
\def\DZ{{{\mathcal{D}}_Z}}

\def\End{{{\operatorname{End}}}}

\def\Ext{{{\operatorname{Ext}}}}
\def\Hom{{{\operatorname{Hom}}}}

\def\Tor{{{\operatorname{Tor}}}}

\def\Max{{{\operatorname{Max}}}}
\def\Max{{{\operatorname{Max}}}}
\def\max{{{\operatorname{max}}}}
\def\KMS{{{\operatorname{KMS}}}}
\def\Per{{{\operatorname{Per}}}}
\def\Out{{{\operatorname{Out}}}}
\def\Aut{{{\operatorname{Aut}}}}
\def\Ad{{{\operatorname{Ad}}}}
\def\Inn{{{\operatorname{Inn}}}}
\def\Int{{{\operatorname{Int}}}}
\def\det{{{\operatorname{det}}}}
\def\exp{{{\operatorname{exp}}}}
\def\nep{{{\operatorname{nep}}}}
\def\sgn{{{\operatorname{sign}}}}
\def\cobdy{{{\operatorname{cobdy}}}}
\def\Ker{{{\operatorname{Ker}}}}
\def\Coker{{{\operatorname{Coker}}}}
\def\Im{{\operatorname{Im}}}

\def\ind{{{\operatorname{ind}}}}
\def\Ind{{{\operatorname{Ind}}}}
\def\id{{{\operatorname{id}}}}
\def\supp{{{\operatorname{supp}}}}
\def\co{{{\operatorname{co}}}}
\def\scoe{{{\operatorname{scoe}}}}
\def\coe{{{\operatorname{coe}}}}
\def\I{{\mathcal{I}}}
\def\Span{{{\operatorname{Span}}}}
\def\event{{{\operatorname{event}}}}
\def\S{\mathcal{S}}

\def\calK{\mathcal{K}}
\def\calP{\mathcal{P}}

\section{Introduction}
The Cuntz--Krieger algebras introduced in \cite{CK} 
provide many interesting examples of Kirchberg algebras 
including the Cuntz algebras $\mathcal{O}_n, n=2,3,\dots,\infty$.
Thanks to the Kirchberg--Phillips' classification theorem and M. Dadarlat's formulas \cite[Corollary 5.10]{Dadarlat2007},
the isomorphism classes of Kirchberg algebras and the homotopy groups of their automorphism groups 
can be understood via their K-groups and KK-groups.
In \cite{MatSogabe},
the authors investigated the homotopy groups of the automorphism groups of Cuntz--Krieger algebras  
by using the Dadarlat's formulas above,
and clarified a relationship between the extension groups 
and the homotopy groups of the automorphism groups 
to show that the strong and weak extension groups as well as  
the homotopy groups are complete invariants to classify Cuntz--Krieger algebras.
In relation to the homotopy groups, 
the second named author introduced a duality, called reciprocality, 
for unital Kirchberg algebras which comes from the study of the  bundles of Kirchberg algebras \cite{Sogabe2022}.
He proved that two unital Kirchberg algebras $\A, \B$ with finitely generated K-groups are 
either isomorphic  or reciprocal if and only if 
their homotopy groups of the automorphism groups are isomorphic 
\cite[Theorem 1.2]{Sogabe2022}.
The authors in \cite{MatSogabe}
showed that the reciprocal algebra of a Cuntz--Krieger algebra can not be realized as 
any Cuntz--Krieger algebra, so that the homotopy groups of the automorphism groups
of Cuntz--Krieger algebras determine the isomorphism classes of the Cuntz--Krieger algebras.

In the present paper,
we will first introduce a hierarchy 
$\mathcal{K}(l, w), \; l\in\Z,\; w\in \{0, 1\}$ 
for unital UCT Kirchberg algebras with finitely generated K-groups (see Section \ref{sec:3.2}).
By using the notion of hierarchy, 
we give a generalization of \cite{MatSogabe} as in Theorem \ref{thm:1.1}.
Let $\mathcal{K}(l, 0)$ (resp. $\mathcal{K}(l, 1)$) 
be the class of unital Kirchberg algebras $\A$
with finitely generated K-groups 
such that the difference of the ranks of the K-groups satisfies 
$l=\rank( \K_0(\A))-\rank(\K_1(\A))$ and the unit $[1_\A]_0$ is a torsion (resp. non-torsion) element in $\K_0(\A)$.
\begin{theorem}[{Theorem \ref{Hei}}] \label{thm:1.1}
The reciprocality gives a bijective correspondence
\[\mathcal{K}(l, w)\ni \A\mapsto \widehat{\A}\in\mathcal{K}(1-l, 1-w),\]
up to isomorphisms for each $l\in\Z$ and $w \in \{0,1\}$.
In particular, if both of two unital Kirchberg algebras $\A, \B$ 
lie in either 
$\mathcal{K}_{>0}:=\bigcup_{l>0,\;w=0, 1}\mathcal{K}(l, w)$ 
or 
$\mathcal{K}_{\leq 0}:=\bigcup_{l\leq 0, w=0, 1}\mathcal{K}(l, w)$,
then the conditions 
$\pi_i(\Aut(\A))\cong \pi_i(\Aut(\B)),\, \, i=1, 2$ 
imply $\A\cong \B$.
\end{theorem}
The Cuntz--Krieger algebras $\OA$ 
lie in $\mathcal{K}_{\leq 0}$ 
and the above theorem clarifies 
the reason why the pair
$(\pi_1(\Aut(\OA)), \pi_2(\Aut(\OA)))$
of the homotopy groups  are complete invariants to classify the Cuntz--Krieger algebras 
(Corollary \ref{hico}, \cite[Theorem 4.6]{MatSogabe}).

In the classification theorem of simple Cuntz--Krieger algebras proved by \cite{Ro}, 
the complete invariant is given by the pair 
$(\K_0(\OA), [1_\OA]_0)$ of the $\K_0$-group 
and the position $[1_\OA]_0$ of the unit $1_{\OA}$ in $\K_0(\OA)$.
The authors  proved in \cite{MatSogabe} 
that the pair $(\Exts^1(\OA), \Extw^1(\OA))$ of the strong and weak extension groups 
 is also a complete invariant.
We define the total extension groups of a separable unital nuclear C*-algebra $\A$ by
\begin{equation*}
\Extt^i(\A) : =\Exts^i(\A)\oplus \Extw^i(\A), \quad i=0, 1.
\end{equation*}
As a main result of the present paper,
we show the following theorem.
\begin{theorem}[{Theorem \ref{MTfg}}]\label{thm:1.2}
Let $\A, \B$ be unital Kirchberg algebras with finitely generated K-groups.
Then they are isomorphic if and only if
\[\Extt^i(\A)\cong \Extt^i(\B),\quad i=0, 1.\]
\end{theorem}
Theorem \ref{thm:1.2} shows that for a unital Kirchberg algebra $\A$
with finitely generated K-groups, 
the pair  
$(\Extt^1(\A), \Extt^0(\A))$ 
of two abelian groups
completely remembers the three data $(\K_0(\A), [1_A]_0, \K_1(\A))$.
Corresponding to R\o rdam's result \cite{Ro} 
classifying $\OA$ by $(K_0(\OA), [1_\OA]_0)$,
we  need only the first total extension group 
$\Extt^1(\OA)$ to determine the isomorphism class of $\OA$.
\begin{corollary}[{Corollary \ref{OAinv}}]
Let $A=[A(i,j)]_{i,j=1}^N, B=[B(i,j)]_{i,j=1}^M$ be irreducible non-permutation matrices
with entries in $\{0,1\}$.
Then the Cuntz--Krieger algebras $\OA$ and $\OB$ are isomorphic 
if and only if one has
\[\Z^{2N}/(I_{2N}-\widehat{A}\oplus A)\Z^{2N}\cong \Z^{2M}/(I_{2M}-\widehat{B}\oplus B)\Z^{2M},\]
where the $N\times N$ (resp. $M\times M)$
matrix $\widehat{A} $ (resp. $\widehat{B}$)
 is defined by
 $\widehat{A}:=A + R_1 -AR_1$ (resp. $\widehat{B}:=B + R_1 -BR_1$), where
 $$R_1:=
 \begin{bmatrix}
 1&\cdots&1\\
 0&\cdots &0\\
 \vdots&\ddots&\vdots\\
0&\cdots&0
\end{bmatrix}.
$$
\end{corollary}
The abelian group $\Z^{2N}/(I_{2N}-\widehat{A}\oplus A)\Z^{2N}$ 
is computed by finding the Smith normal form of the $2N\times 2N$
matrix $I_{2N}-\widehat{A}\oplus A$,
which is obtained by an easy algorithm computing 
the greatest common divisers of minor determinants.
Thus, the above corollary provides a useful algorithm to distinguish between 
two Cuntz--Krieger algebras
$\OA$ and $\OB$.


\section{Preliminaries}
\subsection{Notation}
In the present paper, 
we basically consider the separable nuclear UCT C*-algebras with finitely generated K-groups.
Let $\mathbb{K}$ 
be the algebra of compact operators on the separable infinite dimensional Hilbert space,
and we write $S:=C_0(0, 1)$. 
For a unital C*-algebra $\A$,
we denote by 
$$
C_\A :=\{a(t)\in C_0(0, 1]\otimes \A \;|\; a(1)\in \mathbb{C}1_\A\}
$$ 
the mapping cone algebra of the unital map $\mathbb{C}\to \A$.
The Puppe sequence 
$S\to S\A\to C_\A\to \mathbb{C}\to \A\to SC_\A\to S$
induces the exact sequence of K-groups
\begin{align}
0\to \K_1(\A)\to \K_0(C_\A)\to \Z\to \K_0(\A)\to \K_1(C_\A)\to 0\label{M}
\end{align}
 (see \cite[Theorem 19.4.3]{Blackadar}).
We denote by 
$\mathcal{O}_n, \; n=2, 3, \dots, \infty$ 
the Cuntz algebras whose K-groups are given by (see \cite{Cuntz81})
\[
\K_0(\mathcal{O}_n)\cong \Z/(n-1)\Z \,\,   \text{ for } \,  n<\infty, \,\,
\K_0(\mathcal{O}_\infty)=\Z,\, \text{ and } \, 
\K_1(\mathcal{O}_n)=0 \, \text{ for }\,  n\leq\infty.
\]
If $\K_0(\A), \K_1(\A)$ are finitely generated,
we write 
\begin{equation}\label{eq:chi}
\chi(\A):=\rank (\K_0(\A))-\rank (\K_1(\A))
\end{equation} 
where the rank of finitely generated abelian group $G$ 
is defined by 
$\rank (G):={\rm dim}_\mathbb{Q} (G\otimes_\Z\mathbb{Q})$. 
For a finitely generated abelian group $G$ and $g\in G$,
we write
\begin{equation*}
w (G, g):=\rank (G)-\rank (G/\Z g)\in \{0, 1\}
\end{equation*}
and one can check that
\begin{equation*}
w(G, g) =
\begin{cases}
0 & \text{ if } g \in {\rm Tor}(G) \\
1 & \text{ if } g \not\in {\rm Tor}(G),
\end{cases}
\end{equation*}
where 
$\Z g$ denotes the subgroup of $G$ generated by $g$ 
and ${\rm Tor}(G)$ is the torsion part of $G$ so that 
$G\cong \Z^{\rank (G)}\oplus {\rm Tor}(G)$.
For the triples $(G_i, g_i, H_i)$ of abelian groups $G_i, H_i$
with $g_i\in G_i, \;i=1, 2$,
the isomorphism 
$(G_1, g_1, H_1)\cong (G_2, g_2, H_2)$ 
means that there exist isomorphisms 
$\varphi : G_1\to G_2, \psi : H_1\to H_2$
of abelian groups satisfying $\varphi (g_1)=g_2$.

For the $\sqK$-group $\sqK(\A, \B)$,
we denote by 
\[\hat{\otimes} : \sqK(\A, \B)\times \sqK(\B, \C)\to \sqK(\A, \C)\]
the Kasparov product and write 
$I_\A:=\sqK({\rm id}_\A)\in \sqK(\A, \A)$ 
where $\sqK(f) \in \sqK(\A, \B)$ is the Kasparov module 
$(f, \B, 0)$ for a $*$-homomorphism $f : \A\to \B$ (see \cite{Blackadar} for detail).

\subsection{$\sqK$-groups, weak and strong extension groups}
We refer to \cite{Blackadar} for the basic facts and notations related to $\sqK$-theory.
For UCT C*-algebras $\A, \B$,
we will frequently use the following short exact sequence
\[
0
\to\bigoplus_{i=0, 1}\Ext_\Z^1(\K_i(\A), \K_{i+1}(\B))
\to \sqK(\A, \B)
\to \bigoplus_{i=0, 1}\Hom (\K_i(\A), \K_i(\B))\to 0
\]
called UCT, which splits unnaturally (see \cite{Blackadar, Brown84, RS}).
We write 
$$\Extw^i(\A):=\sqK(\A, S^i), \quad i=0, 1,$$
where $S^0, S^1$ denote $\mathbb{C}, S$ respectively.
For a unital separable nuclear C*-algebra $\A$, 
it is well-known that the group $\Extw^1(\A)$ 
is identified with the weak extension group $\Extw(\A)$
consisting of the weak unitary equivalence classes of unital Busby invariants 
$\tau : \A\to\mathcal{Q}(\mathbb{K})$ 
(see \cite[Chapter 15, Proposition 15.14.2]{Blackadar}).
There is also another group consisting of the strong unitary equivalence classes of the unital Busby invariants, 
called the strong extension groups written $\Exts(\A)$.
G. Skandalis in \cite{Skandalis} 
clarifies the relation of the strong extension groups to $\sqK$-groups.
Following Skandalis' result for  unital C*-algebras,
we write 
\[\Exts^i(\A):=\sqK(C_\A, S^{i+1}), \qquad  i=0, 1,
\]
so that $\Exts^1(\A)$ is identified with $\Exts(\A)$ by \cite{Skandalis}.
For example,
the UCT shows the following isomorphisms for a unital C*-algebra $\A$ 
with finitely generated K-groups:
\begin{align*}
\Exts^1(\A)
&\cong \sqK(C_\A, \mathbb{C})\\
&\cong \Z^{\rank(\K_0(C_\A))}\oplus {\rm Tor}(\K_1(C_\A))\\
&=\Z^{\chi (C_\A)+\rank(\K_1(C_\A))}\oplus {\rm Tor}(\K_1(C_\A))\\
&=\Z^{1-\chi (\A)+\rank (\K_1(C_\A))}\oplus {\rm Tor}(\K_1(C_\A)),
\end{align*}
where the last equality follows from the exact sequence (\ref{M}).
Applying Puppe sequence 
$S\to S\A\to C_\A\to\mathbb{C}\to \A \to SC_\A \to S$
for $\sqK(-, \mathbb{C})$,
one has the cyclic six term exact sequence of extension groups 
(see \cite[Theorem 2.3]{Skandalis}):
\begin{equation}
\xymatrix{
\Z\ar[r]^{\iota_\A}&\Exts^1(\A)\ar[r]&\Extw^1(\A)\ar[d]\\
\Extw^0(\A)\ar[u]&\Exts^0(\A)\ar[l]&0,\ar[l]
}\label{E6}
\end{equation}
and we write $\iota_\A:=\iota_\A(1)\in \Exts^1(\A)$ by abuse of notation.

\section{Spanier--Whitehead $\K$-duality and reciprocal duality}
\subsection{Spanier--Whitehead $\K$-duality}
We briefly recall the Spanier--Whitehead K-duality following \cite{KS, KP}.
\begin{definition}
Let $\A, D(\A)$ be separable C*-algebras.
They are called Spanier--Whitehead K-dual if there are elements
\[\mu\in \sqK(\mathbb{C}, \A\otimes D(\A)),
\quad 
\nu\in \sqK(D(\A)\otimes \A, \mathbb{C})\]
satisfying
\[
(\mu\otimes I_\A)\hat{\otimes}(I_\A\otimes \nu)=I_\A,
\quad
(I_{D(\A)}\otimes \mu)\hat{\otimes}(\nu\otimes I_{D(\A)})=I_{D(\A)}.
\]
The pair $(\mu, \nu)$ is called the duality classes.
\end{definition}
If a C*-algebra $\A$ has a C*-algebra $D(\A)$ 
such that they are 
Spanier--Whitehead $\K$-dual,
then it is said to be dualizable.

Let $\A, \B$ be dualizable C*-algebras with their duality classes 
\begin{gather*}
\mu_\A\in \sqK(\mathbb{C}, \A\otimes D(\A)), \quad\nu_\A\in \sqK(D(\A)\otimes \A, \mathbb{C}),\\
\mu_\B\in \sqK(\mathbb{C}, \B\otimes D(\B)), \quad\nu_\B\in \sqK(D(\B)\otimes \B, \mathbb{C}),
\end{gather*}
respectively.
For the natural inclusions 
\begin{gather*}
i_\A :\A\otimes D(\A)\ni a\otimes d \mapsto 
(a\otimes 0) \oplus (d\oplus 0) \in (\A\oplus \B)\otimes (D(\A)\oplus D(\B)),\\
i_\B : \B\otimes D(\B) \ni b\otimes d \mapsto 
(0\oplus b) \otimes (0\oplus d) \in (\A\oplus \B)\otimes (D(\A)\oplus D(\B))
\end{gather*}
and the natural projections
\begin{gather*}
p_\A : (D(\A)\oplus D(\B))\otimes (\A\oplus \B)\ni (d_1\oplus d_2)\otimes (a\oplus b)  
\mapsto d_1\otimes a \in D(\A)\otimes \A,\\
p_\B : (D(\A)\oplus D(\B))\otimes (\A\oplus \B)\ni (d_1\oplus d_2)\otimes (a\oplus b)
\mapsto d_2\otimes b \in D(\B)\otimes \B,
\end{gather*}
we write
\begin{gather*}
\overline{\mu_\A}:=\mu_\A\hat{\otimes}\sqK(i_\A),\quad \overline{\nu_\A}:=\sqK(p_\A)\hat{\otimes}\nu_\A,\\
\overline{\mu_\B}:=\mu_\B\hat{\otimes}\sqK(i_\B),\quad \overline{\nu_\B}:=\sqK(p_\B)\hat{\otimes}\nu_\B.
\end{gather*}

\begin{lemma}\label{dsum}
The elements 
\begin{gather*}
\overline{\mu_\A}+\overline{\mu_\B}\in \sqK(\mathbb{C}, (\A\oplus \B)\otimes (D(\A)\oplus D(\B))),\\
\overline{\nu_\A}+\overline{\nu_\B}\in \sqK((D(\A)\oplus D(\B))\otimes (\A\oplus \B), \mathbb{C})
\end{gather*}
are duality classes.
In particular, $\A\oplus \B$ and $D(\A)\oplus D(\B)$ are Spanier--Whitehead $\K$-dual.
\end{lemma}
\begin{proof}
We prove the equality
\[((\overline{\mu_\A}+\overline{\mu_\B})\otimes I_{\A\oplus \B})
\hat{\otimes}
(I_{\A\oplus \B}\otimes(\overline{\nu_\A}+\overline{\nu_\B}))
=I_{\A\oplus \B}.\]
Let $j_\A : \A\to \A\oplus \B$ and $q_\A : \A\oplus \B\to \A$ 
denote the natural inclusion and projection respectively.
The direct computation yields
\begin{align*}
&(\overline{\mu_\A}\otimes I_{\A\oplus \B})\hat{\otimes}(I_{\A\oplus \B}\otimes \overline{\nu_\A})\\
=&(\mu_\A\otimes I_{\A\oplus \B})
\hat{\otimes}\sqK(\id_{\A\oplus \B}\otimes p_\A\circ i_\A\otimes\id_{\A\oplus \B})\hat{\otimes}(I_{\A\oplus \B}\otimes \nu_\A)\\
=&(\mu_\A\otimes I_{\A\oplus \B})\hat{\otimes}\sqK(j_\A\otimes\id_{D(\A)}\otimes q_\A)\hat{\otimes}(I_{\A\oplus \B}\otimes \nu_\A)\\
=&\sqK(q_\A)\hat{\otimes}(\mu_\A\otimes I_\A)\hat{\otimes}(I_\A\otimes \nu_\A)\hat{\otimes}\sqK(j_\A)\\
=&\sqK(j_\A\circ q_\A),
\end{align*}
and the same computation shows
\[
(\overline{\mu_\B}\otimes I_{\A\oplus \B})
\hat{\otimes}
(I_{\A\oplus \B}\otimes \overline{\nu_\B})=\sqK(j_\B\circ q_\B)
\]
for the natural inclusion $j_\B : \B\to \A\oplus \B$ 
and projection $q_\B : \A\oplus \B\to \B$.
Since the composition
\[
(\id_{\A\oplus \B}\otimes p_\A)\circ(i_\B\otimes \id_{\A\oplus \B})=j_B\otimes (q_{D(\A)}\circ j_{D(\B)})\otimes q_A\]
is the zero map,
one has 
$
(\overline{\mu_\B}\otimes I_{\A\oplus \B})\hat{\otimes}(I_{\A\oplus \B}\overline{\nu_\A})=0,
$
and similarly 
$
(\overline{\mu_\A}\otimes I_{\A\oplus \B})\hat{\otimes}(I_{\A\oplus \B}\otimes\overline{\nu_\B})=0.
$
Let $\phi_0$ be the map
\[\phi_0 : \A\oplus \B\ni (a, b)\mapsto 
\begin{bmatrix}
(a, 0)&0\\
0&(0, b)
\end{bmatrix}
\in \mathbb{M}_2(\A\oplus \B)\]
which is homotopic to
\[\phi_1 : \A\oplus \B\ni (a, b)\mapsto 
\begin{bmatrix}
(a, b)&0\\
0&0
\end{bmatrix}\in \mathbb{M}_2(\A\oplus \B).
\]
We thus have
\begin{align*}
&((\overline{\mu_\A}+\overline{\mu_\B})\otimes I_{\A\oplus \B})
\hat{\otimes}(I_{\A\oplus \B}\otimes(\overline{\nu_\A}+\overline{\nu_\B}))\\
=&\sqK(j_\A\circ q_\A)+0+0+ \sqK(j_\B\circ q_\B)\\
=&\sqK(\phi_0)=\sqK(\phi_1)=I_{\A\oplus \B},
\end{align*}
The same argument shows the other equality
\[(I_{D(\A)\oplus D(\B)}\otimes(\overline{\mu_\A}+\overline{\mu_\B}))
\hat{\otimes}((\overline{\nu_\A}+\overline{\nu_\B})\otimes I_{D(\A)\oplus D(\B)})
=I_{D(\A)\oplus D(\B)}.\]
\end{proof}
\begin{corollary}\label{cor:dsum}
Every separable UCT C*-algebra $\A$ with finitely generated K-groups is 
$\sqK$-equivalent to a finite direct sum of the following dualizable C*-algebras
\[\mathbb{C}, \quad S,\quad  \mathcal{O}_n,\quad  S\mathcal{O}_n, \; n=2, 3, \dots,\]
and hence it is dualizable by the above lemma (see also \cite{KP, KS, PennigSogabe}).
\end{corollary}
Note that duals of 
$\mathbb{C}, S, \mathcal{O}_n, S\mathcal{O}_n$ are given by
$\mathbb{C}, S, S\mathcal{O}_n, \mathcal{O}_n$, respectively 
(see \cite{KP, KS, PennigSogabe}).
Lemma \ref{dsum} and Corollary \ref{cor:dsum}
show that a dual $D(\A)$
of a C*-algebra $\A$ $\sqK$-equivalent to
$\mathbb{C}^a\oplus S^b\oplus \mathcal{O}_n^c\oplus (S\mathcal{O}_m)^d\oplus\cdots$ is given by
\[D(\A)=\mathbb{C}^a\oplus S^b\oplus (S\mathcal{O}_n)^c\oplus \mathcal{O}_m^d\oplus\cdots,\]
and one has
\[D(D(\A))\sim_{\sqK} \A,\quad \sqK(\A, S^i)\cong \sqK(S^i, D(\A)),\]
\[
\K_0(D(\A))=\Z^{\rank(\K_0(\A))}\oplus {\rm Tor}(\K_1(\A)),\quad \K_1(D(\A))
=\Z^{\rank(\K_1(\A))}\oplus {\rm Tor}(\K_0(\A)).
\]

In topology, the Euler characteristic is given by
a specific cohomology class called the Euler class.
The following proposition shows that the number $\chi(\A)$
defined by \eqref{eq:chi} is given by a $\sqK$-class.
\begin{proposition}
Let $\A$ be a UCT C*-algebra with finitely generated K-groups.
For the duality classes 
$\mu\in \sqK(\mathbb{C}, \A\otimes D(\A)), \nu\in \sqK(D(\A)\otimes \A, \mathbb{C})$,
and the flip isomorphism 
$\sigma_{\A, D(\A)} : \A\otimes D(\A)\to D(\A)\otimes \A$
defined by
$\sigma_{\A, D(\A)}(a\otimes d) = d\otimes a$
for
$a \otimes d \in \A\otimes D(\A)$,
we have 
\[
\chi (\A)
=\mu\hat{\otimes}\sqK(\sigma_{\A, D(\A)})\hat{\otimes}\nu
\in \sqK(\mathbb{C}, \mathbb{C}) =\Z.
\]
\end{proposition}
\begin{proof}
If $\A$ is one of the following building blocks
\[\mathbb{C},\; S,\; \mathcal{O}_n,\; S\mathcal{O}_n,\;\; n=2,3, \dots,\]
their duals are given by 
\[\mathbb{C},\; S,\; S\mathcal{O}_n,\; \mathcal{O}_n,\;\; n=2,3, \dots.\]
Let 
$\beta\in \sqK(\mathbb{C}, S\otimes S)$ 
be the Bott element (cf. \cite[19.2.5]{Blackadar}).
For $\A=\mathbb{C}$,
the assertion holds obviously.
For $\A=S$,
 the pair $(\beta, \beta^{-1})$ gives  duality classes such that 
 \begin{align*}
\beta\hat{\otimes}\sqK(\sigma_{S, S})\hat{\otimes}\beta^{-1}
=\beta\hat{\otimes}(-I_{S^2})\hat{\otimes}\beta^{-1}
=-1=\chi (S).
\end{align*}
For $\A=\mathcal{O}_n, S\mathcal{O}_n$,
the composition $\mu\hat{\otimes}\sqK(\sigma_{\A, D(\A)})\hat{\otimes}\nu$ 
lies in the image of the group homomorphism
\[\sqK(\mathbb{C}, \mathcal{O}_n\otimes (S\mathcal{O}_n))
\xrightarrow{\hat{\otimes}\nu}
\sqK(\mathbb{C}, \mathbb{C})=\Z\]
which is the zero map because 
$\sqK(\mathbb{C}, \mathcal{O}_n\otimes (S\mathcal{O}_n))\cong \Z/(n-1)\Z$.
Hence one has 
$\mu\hat{\otimes}\sqK(\sigma_{\A, D(\A)})\hat{\otimes}\nu
=0=\chi(\A)$.

Now let $\A, \B$ be C*-algebras with finitely generated $\K$-groups
such that
\[
\chi(\A)=\mu_\A\hat{\otimes}\sqK(\sigma_{\A, D(\A)})\hat{\otimes}\nu_\A, 
\quad 
\chi(\B)=\mu_\B\hat{\otimes}\sqK(\sigma_{\B, D(\B)})\hat{\otimes}\nu_\B
\]
holds for the duality classes $(\mu_\A, \nu_\A),\; (\mu_\B, \nu_\B)$ (i.e., $\A, \B$ satisfy the assertion of the proposition).
As in Lemma \ref{dsum},
the pair 
$(\overline{\mu_\A}+\overline{\mu_\B}, \overline{\nu_\A}+\overline{\nu_\B})$
gives 
duality classes for 
$\A\oplus \B$ and $D(\A)\oplus D(\B)$ 
such that 
\begin{align*}
&(\overline{\mu_\A}+\overline{\mu_\B})\hat{\otimes}\sqK(\sigma_{(\A\oplus \B), (D(\A)\oplus D(\B))})
\hat{\otimes}(\overline{\nu_\A}+\overline{\nu_\B}) \\
=&\mu_\A\hat{\otimes}\sqK(\sigma_{\A, D(\A)})\hat{\otimes}\nu_\A
+\mu_\B\hat{\otimes}\sqK(\sigma_{\B, D(\B)})\hat{\otimes}\nu_\B.\\
=&\chi (\A)+\chi (\B)\\
=&\chi (\A\oplus \B).
\end{align*}
Thus,  the direct sum $\A\oplus \B$ 
also satisfies the assertion of the proposition.
By Corollary \ref{cor:dsum},
we have the desired assertion.
\end{proof}

\subsection{Reciprocal duality}\label{RK}
We recall some basic properties of the reciprocality introduced in \cite{Sogabe2022}.
\begin{definition}
Let $\A, \B$ 
be unital separable UCT C*-algebras with finitely generated K-groups.
They are said to be reciprocal if the following $\sqK$-equivalences hold:
\[D(C_\A)\sim_{\sqK}\B,\quad D(C_{\B})\sim_{\sqK} \A.\]
In this case, we say that $\B$ is reciprocal to $\A$, or 
$\A$ is reciprocal to $\B$.
\end{definition}
\begin{proposition}[{\cite[Theorem 1.2, Remark 1.6, Section 3.3]{Sogabe2022}}]
For every unital Kirchberg algebra $\A$ 
with finitely generated K-groups,
there is a unique  unital Kirchberg algebra $\widehat{\A}$ up to isomorphism
which is reciprocal to $\A$.
In particular, the reciprocal Kirchberg algebras satisfy 
$(\widehat{\widehat{\A}})\cong \A$.
\end{proposition}
For a separable unital  C*-algebra $\A$ with finitely generated $\K$-groups,
we write
$w(\A) := w(\K_0(\A), [1_\A]_0).$
\begin{lemma}\label{lem:wchi}
Let
$\A$ 
be a unital Kirchberg algebra with finitely generated K-groups.
\begin{enumerate}
\renewcommand{\theenumi}{(\roman{enumi})}
\renewcommand{\labelenumi}{\textup{\theenumi}}
\item
$w(\A) + w(\whatA) =1.$
\item
$\chi(\A) + \chi(\whatA) =1.$
\end{enumerate}
\end{lemma} 
\begin{proof}
(i) We have
\begin{align*}
 w(\K_0(\widehat{\A}), [1_{\widehat{\A}}]_0)
=&\rank(\K_0(\widehat{\A}))-\rank(\K_1(C_{\widehat{\A}}))\\
=&\rank(\K_0(D(C_\A)))-\rank(\K_1(D(\A))\\
=&\rank(\K_0(C_\A))-\rank(\K_1(\A))\\
=&1-(\rank(\K_0(\A))-\rank(\K_1(C_\A)))\\
=&1-w(\K_0(\A), [1_\A]_0) 
\end{align*}
so that
\begin{equation}\label{w}
w(\A) +w(\whatA)
=1.
\end{equation}
(ii)
The exact sequence \eqref{M} shows
\begin{equation}\chi (\A)+\chi (\widehat{\A})=1, \label{eu}
\end{equation}
in a similar way.
\end{proof} 

To clarify the difference between the Spanier--Whitehead $\K$-duality 
and 
the reciprocality,
we introduce the notation $K^s_i(-), K^w_i(-), \, i=0,1$ 
for a unital C*-algebra $\A$ such as 
\[K^w_i(\A):=K_i(\A), \quad K^s_i(\A):=K_{i+1}(C_\A).\]
The Spanier--Whitehead $\K$-duality 
is understood as a Poincare-like duality between $K^w_i(-)$ and $\Extw^i(-)$ 
such that  
\begin{equation*}
\K^w_i(D(\A))\cong \Extw^i(\A), \qquad \Extw^i(D(\A))\cong \K^w_i(\A).
\end{equation*}
In contrast to the Spanier--Whitehead $\K$-duality,
the reciprocality is a duality between 
$K^w_i(-)$ (resp. $K^s_i(-)$) and $\Exts^{i+1}(-)$ (resp. $\Extw^{i+1}(-)$).

\begin{proposition}[{cf. \cite[Proposition 1.5]{Sogabe2022}}]\label{dgi}
\begin{enumerate}
\renewcommand{\theenumi}{(\roman{enumi})}
\renewcommand{\labelenumi}{\textup{\theenumi}}
\item
For a unital Kirchberg algebra $\A$ with finitely generated K-groups,
the following holds
\[\Exts^{i+1}(\A)\cong K^w_{i}(\widehat{\A}),\quad \Extw^{i+1}(\A)\cong K^s_{i}(\widehat{\A}), \quad i=0, 1.\]
In particular, we have an isomorphism
\begin{align}
(\Exts^1(\A), \iota_\A, \Exts^0(\A))\cong (\K_0(\widehat{\A}), [1_{\widehat{\A}}]_0, \K_1(\widehat{\A})), \label{ste}
\end{align}
and the exact sequence of the extension groups \eqref{E6} 
is identified with the mapping cone sequence
\eqref{M} for the unital C*-algebra $\widehat{\A}$.
\item
If two unital separable UCT C*-algebras $\A, \B$ satisfy
\begin{equation}\label{eq:ExfsKw}
\Exts^{i+1}(\A)\cong K^w_i(\B),\quad \Extw^{i+1}(\A)\cong K^s_i(\B),\quad i=0, 1,
\end{equation} 
then $\A$ and $\B$ are reciprocal.
\item
For two unital Kirchberg algebras
$\A,\B$ with finitely generated $\K$-groups,
we have $\B = \whatA$ if and only if  \eqref{eq:ExfsKw} holds.
\end{enumerate}
\end{proposition}
\begin{remark}
We will frequently use \cite[Proposition 1.5]{Sogabe2022} 
stating that for finitely generated abelian groups 
$G_i, g_i\in G_i$ satisfying 
$G_1\cong G_2$ { and} $G_1/\Z g_1\cong G_2/\Z g_2$,
there exists an isomorphism $\varphi : G_1\to G_2$
satisfying $\varphi(g_1) = g_2$.
In the above proposition,
one has 
\begin{align*}
\K_0(\widehat{\A}) & \cong \Exts^1(\A),\\
\K_0(\widehat{\A})/\Z [1_{\widehat{\A}}]_0 \cong \K_1(C_{\widehat{\A}})
 & \cong \Extw^1(\A)\cong \Exts^1(\A)/\Z \iota_\A.
 \end{align*}
Thus, one can check the isomorphism (\ref{ste}).
\end{remark}


\section{Hierarchy of Kirchberg algebra with finitely generated K-groups}
Focusing on the values of $\chi(\A)$ and $w(\A)$, 
we will introduce a hierarchy for the class of unital Kirchberg algebras with finitely generated K-groups and this hierarchy gives a general picture to understand the results in \cite{MatSogabe}. 
\subsection{The relations of several invariants}
Before going to the hierarchy, we clarify the relationships between several invariants to determine the data $(\K_0(\A), [1_\A]_0, \K_1(\A))$.
We say that two invariants {for unital separable UCT algebras with finitely generated K-groups} are equivalent if one is completely determined by {the other}.
For example,
the isomorphism classes of 
$(\K_0(\A), [1_\A]_0)$ and $(\K_0(\A), \K_1(C_\A))$ 
are equivalent invariants if K-groups of $\A$ are finitely generated, and we write
\[(\K_0(\A), [1_\A]_0)\sim_e (\K_0(\A), \K_1(C_\A)).\]
It is easy to see that $(\K_0(\A), [1_\A]_0)$ determines 
$(\K_0(\A), \K_1(C_\A))$ by $\K_1(C_\A)\cong \K_0(\A)/\Z [1_\A]_0$ and the converse direction follows from \cite[Appendix, Proposition 1.5]{Sogabe2022} (see also Lemma \ref{star}).
\begin{remark}
The above equivalence and \cite[Proposition 1.5]{Sogabe2022} never hold if we drop the assumption ``finitely generated abelian group''.
For example,
the both of $(\bigoplus_1^\infty \Z, (1, 0, 0,\dots))$ and $(\bigoplus_1^\infty \Z, (0, 0, 0,\dots))$ gives the same  isomorphism class of the pair of the group and its quotient
$(\bigoplus_1^\infty \Z, \bigoplus_1^\infty \Z)$.
Thus, the pair $((\bigoplus_1^\infty \Z), (\bigoplus_1^\infty \Z)/\Z g)$ 
does not remember $(\bigoplus_1^\infty \Z, g)$
for $g = (1, 0, 0,\dots), (0, 0, 0,\dots).$
\end{remark}
The following is a list of equivalent invariants.
{
\begin{lemma}\label{eq}
Let $\A$ be a unital Kirchberg algebra with finitely generated K-groups.
Then, we have the following equivalences:
\begin{align*}
&(\K_0(\A), [1_\A]_0, \K_1(\A))
&\sim_e\quad&
(\K_0(\A), \K_1(C_\A), \chi(\A), {\rm Tor}(\K_1(\A)))\\
&&\sim_e\quad&
(\K_0(\widehat{\A}), \K_1(C_{\widehat{\A}}), \chi(\widehat{\A}), {\rm Tor}(\K_1(\widehat{\A})))\\
&&\sim_e\quad&(\Exts^1(\A), \Extw^1(\A), \chi(\A), {\rm Tor} (\Extw^0(\A))),
\end{align*}

\end{lemma}
}
\begin{proof}
We only show the second equivalence above: 
\[(\K_0(\A), \K_1(C_\A), \chi(\A), {\rm Tor}(\K_1(\A)))\sim_e (\K_0(\widehat{\A}), \K_1(C_{\widehat{\A}}), \chi(\widehat{\A}), {\rm Tor}(\K_1(\widehat{\A})))\]
because other equivalences follow from similar arguments using the UCT and results in Section 2.

Direct computation with the UCT and (\ref{ste}) yield 
\[\K_0(\widehat{\A})\cong \Exts^1(\A)\cong\Z^{1-\chi (\A)+\rank(\K_1(C_\A))}\oplus {\rm Tor}\; (\K_1(C_\A)).\]
\[\K_1(C_{\widehat{\A}})\cong \K_1(D(\A))\cong \Z^{\rank (\K_0(\A))-\chi (\A)}\oplus {\rm Tor}\; (\K_0(\A)),\]
so that
\begin{align*}
&(\K_0(\widehat{\A}), \K_1(C_{\widehat{\A}}), \chi (\widehat{\A}))\\
=&(\Z^{1-\chi (\A)+\rank(\K_1(C_\A))}\oplus {\rm Tor}\; (\K_1(C_\A)), \Z^{\rank(\K_0(\A))-\chi (\A)}\oplus {\rm Tor}(\K_0(\A)), 1-\chi (\A)).
\end{align*}
Since $\A\cong\widehat{\widehat{\A}}$, we have $(\K_0(\A), \K_1(C_\A), \chi(\A))\sim_e (\K_0(\widehat{\A}), \K_1(C_{\widehat{\A}}), \chi (\widehat{\A}))$.
Furthermore,
the direct computation yields
\begin{align*}
{\rm Tor}(\K_1(\widehat{\A}))=&{\rm Tor}(\Z^{1-w(\K_0(\widehat{\A}), [1_{\widehat{\A}}]_0)}\oplus \K_1(\widehat{\A}))\\
=&{\rm Tor}(\K_0(C_{\widehat{\A}}))\\
\cong&{\rm Tor}(\K_0(D(\A)))\\
\cong&{\rm Tor}(\K_1(\A)).
\end{align*}
\end{proof}
\begin{theorem}\label{extcp}
{The triple $(\Exts^1(\A), \Extw^1(\A), \Extw^0(\A))$ for a unital Kirchberg algebra $\A$ with finitely generated K-groups is a complete invariant to determine the isomorphism class of $\A$,}
and the range of triple is the following:
\begin{align*}
& \{(\Exts^1(\A), \Extw^1(\A), \Extw^0(\A)) \, \mid \, 
\A : \text{unital Kirchberg algebra with finitely generated K-groups}\}\\
= &\{(G_1,  G_1/\Z g,  G_0 \oplus \Z^{1-w(G_1,g)}) \mid
   G_0, G_1: \text{finitely generated abelian groups},
   g\in  G_1 \}.
\end{align*}

\end{theorem}
\begin{proof}
Let $\A, \B$ be unital Kirchberg algebras with finitely generated K-groups.
Suppose that there exists an isomorphism 
\begin{equation}(\Exts^1(\A), \Extw^1(\A), \Extw^0(\A))\cong (\Exts^1(\B), \Extw^1(\B), \Extw^0(\B)).\label{eiso}\end{equation}
Since $\chi(\A)=\rank(\Extw^0(\A))-\rank(\Extw^1(\A))$,
the isomorphism (\ref{eiso}) implies $\chi(\A)=\chi(\B)$ and ${\rm Tor}(\Extw^0(\A))\cong {\rm Tor}(\Extw^0(\B))$.
Thus, Lemma \ref{eq} shows $$(\K_0(\A), [1_\A]_0, \K_1(\A))\cong (\K_0(\B), [1_\B]_0, \K_1(\B))$$
and the triple $(\Exts^1(\A), \Extw^1(\A), \Extw^0(\A))$ provides a complete invariant.

For an arbitrary triple $(G_1, g, G_0)$ of finitely generated abelian groups $G_1, G_0$ and an element $g\in G_1$, one may find a unital Kirchberg algebra $\widehat{A}$ satisfying $(G_1, g, G_0)\cong (\K_0(\widehat{\A}), [1_{\widehat{\A}}]_0, \K_1(\widehat{\A}))$.
Since $(\Exts^1(\A), \iota_\A, \Exts^0(\A))\cong(\K_0(\widehat{\A}), [1_{\widehat{\A}}]_0, \K_1(\widehat{\A}))$, 
the six term exact sequence (\ref{E6}) implies that the range of $(\Exts^1(\A), \Extw^1(\A), \Extw^0(\A))$ is represented by $(G_1, G_1/\Z g, G_0\oplus \Z^{1-w(G_1, g)})$.
\end{proof}

\subsection{A hierarchy of unital Kirchberg algebras} \label{sec:3.2}
Let $\mathcal{K}(l, w)$ for $(l, w)\in \Z\times \{0, 1\}$ 
be the class of unital Kirchberg algebras $\A$
with finitely generated K-groups satisfying 
$\chi(\A)=l,\; w(\A)=w$.
We write $$\mathcal{K}_{>0}:=\bigcup_{l>0, w=0, 1} \mathcal{K}(l, w),\;\;\mathcal{K}_{\leq 0}: =\bigcup_{l\leq 0, w=0, 1}\mathcal{K}(l, w).$$ 
\begin{theorem}\label{Hei}
\begin{enumerate}
\renewcommand{\theenumi}{(\roman{enumi})}
\renewcommand{\labelenumi}{\textup{\theenumi}}
\item The reciprocality gives a {bijective} correspondence
\[\mathcal{K}(l, w)\ni \A\mapsto \widehat{\A}\in \mathcal{K}(1-l, 1-w).\]
\item If both of the two algebras 
$\A, \B$ belong to either $\mathcal{K}_>0$ or $\mathcal{K}_{\leq 0}$, 
they are isomorphic if and only if 
{$\pi_i(\operatorname{Aut}(\A))\cong\pi_i(\operatorname{Aut}(\B)),\,i=1, 2,$\;}
where $\pi_i(\operatorname{Aut}(\A))$ 
is the $i$th homotopy group of the automorphism group 
$\operatorname{Aut}(\A)$ of $\A$.
\end{enumerate}
\end{theorem}
\begin{proof}
The assertion (i) follows from the equations (\ref{w}) and (\ref{eu}).
By \cite[Theorem 1.2]{Sogabe2022},
either $\A\cong \B$ or $\A\cong \widehat{\B}$ must hold,
and (i) shows that $\A\cong\widehat{\B}$ does not happen.
This shows (ii).
\end{proof}
The locations of $\A$ and $\widehat{\A}$ 
in the hierarchy given by $\K(l, w),\, (l, w)\in \Z\times \{0, 1\}$ 
are described in the table below.
\[\begin{array}{|c|ccccc|ccccc|}
\hline
&&&\chi>0&&&&&\chi\leq 0&&\\
\hline
&\cdots&\vline& \chi=l&\vline&\cdots&\cdots&\vline&\chi=1-l&\vline&\cdots\\
\hline
 w=0&\cdots&\vline&\A&\vline&\cdots&\cdots&\vline&\widehat{\B}&\vline&\cdots\\
\hline
 w=1&\cdots&\vline&\B&\vline&\cdots&\cdots&\vline&\widehat{\A}&\vline&\cdots\\
\hline
\end{array}\]
{
The Cuntz algebras 
$\mathcal{O}_n, \; n=2, 3, \dots,\infty$, \; 
simple Cuntz--Krieger algebras $\OA$, 
and the Kirchberg algebra 
$\mathcal{P}_\infty$ $\sqK$-equivalent to $S$ 
lie in the following hierarchy:
\begin{gather*}
\mathcal{P}_\infty\in\mathcal{K}(-1, 0),
\quad
\mathcal{O}_2, \mathcal{O}_3,\dots\in \mathcal{K}(0, 0), 
\quad
\mathcal{O}_\infty\in\mathcal{K}(1, 1),\\
\OA\in\mathcal{K}(0, 0)\cup\mathcal{K}(0, 1).
\end{gather*}
We have the following corollaries.
\begin{corollary}
Let $\A, \B$ be unital Kirchberg algebras with finitely generated K-groups.
Then the following three conditions are equivalent:
\begin{enumerate}
\renewcommand{\theenumi}{(\roman{enumi})}
\renewcommand{\labelenumi}{\textup{\theenumi}}
\item $\A\cong \B$.
\item
$\pi_i(\operatorname{Aut}(\A))\cong \pi_i(\operatorname{Aut}(\B)),\, \, i=1, 2$ 
and $\chi (\A)=\chi (\B)$.
\item 
$\pi_i(\operatorname{Aut}(\A))\cong \pi_i(\operatorname{Aut}(\B)),\, \, i=1, 2$ 
and $w(\A)=w(\B)$.
\end{enumerate}
\end{corollary}

\begin{corollary}[{\cite[Theorem 4.8, Corollary 5.2]{MatSogabe}}]\label{hico}
\;
\begin{enumerate}
\renewcommand{\theenumi}{(\roman{enumi})}
\renewcommand{\labelenumi}{\textup{\theenumi}}
\item Two simple Cuntz--Krieger algebras
 $\OA$ and $\OB$  are isomorphic if and only if 
 $$\pi_i(\operatorname{Aut}(\OA))\cong \pi_i(\operatorname{Aut}(\OB)), \;\; i=1, 2.$$
\item For a simple Cuntz--Krieger algebra $\OA$, 
the triple $(\Exts^1(\OA), \Extw^1(\OA), \Extw^0(\OA))$ 
is equivalent to $(\Exts^1(\OA), \Extw^1(\OA))$ 
whose possible range is given by
\begin{align*}
&\{(\Exts^1(\OA), \Extw^1(\OA)) \;|\; \OA : \text{simple Cuntz--Krieger algebra}\}\\
=&\{(\Exts^1(\B), \Extw^1(\B)\; |\; \B \in\mathcal{K}(0, 0)\cup\mathcal{K}(0, 1), {\rm Tor}(\K_1(\B))=0\}\\ 
=&\{(\K_0(\A), \K_1(C_\A))\; |\; \A\in \mathcal{K}(1, 1)\cup \mathcal{K}(1, 0), {\rm Tor}(\K_1(\A))=0\}\\
=&\{(G,  G/\Z g) \mid G : \text{finitely generated abelian group with } \rank(G)\geq 1, g\in  G \}.
\end{align*}
\end{enumerate}
\end{corollary}

\section{Main result}

In this section,
we first give some combinatorial argument in Proposition \ref{comb},
and we will show our main result (Theorem \ref{MTfg}) as an application of the argument.

Let $G, H$ be finitely generated abelian groups and let $g$ (resp. $h$) be an element of $G$ (resp. $H$).
It is proved in \cite[Proposition 1.5, Appendix]{Sogabe2022} that there is an isomorphism $\varphi : G\cong H$ with $\varphi(g)=h$ if one has two isomorphisms $G\cong H$ and $G/\Z g\cong H/\Z h$.
We will explain that the isomorphisms 
$G\cong H, G/\Z g\cong H/\Z h$ follow from an isomorphism  
$G\oplus (G/\Z g)\cong H\oplus (H/\Z h)$ (Proposition \ref{comb}, see also Remark \ref{rco}).

We use the same notation as in \cite[Section 2.5, Appendix]{Sogabe2022}.
By the fundamental theorem of finitely generated abelian groups, one has the following unique presentation:
\[G=\Z^{\rank(G)}\oplus \bigoplus_{p: \text{prime}} G(p),\;\;\;\text{where}\;\; G(p)=0\; \text{or}\]\[ G(p)=\Z/{p^{n_1}}\Z \oplus \Z/p^{n_2}\Z\oplus \dots \oplus \Z/p^{n_k}\Z,\;\; n_1\geq n_2\geq \dots\geq n_k\geq 1.\]
For $G(p)\not=0$, we introduce a multi-set 
(i.e., a collection of positive integers)
\[I(G(p)):=\{n_1, n_2, \dots , n_k\}.\]
The fundamental theorem shows that $I(G(p))$ completely determines $G(p)$.
For example, we have
\[I(\Z/p^2\Z\oplus\Z/p\Z\oplus \Z/p\Z)=\{2, 1, 1\},\quad I(\Z/p^2\Z\oplus \Z/p\Z)=\{2, 1\}.\]
For two finitely generated abelian groups $G, H$ and a prime number $p$,
we carefully define the intersection $I(G(p))\cap I(H(p))$ as the maximal subset of both $I(G(p))$ and $I(H(p))$ such that $I(G(p))\backslash (I(G(p))\cap I(H(p)))$ and $I(H(p))\backslash (I(G(p))\cap I(H(p)))$ do not share any common integers.
Since $|I(G(p))|, |I(H(p))|<\infty$, the above definition is well-defined.
For example, one has
\[I(\Z/p^2\Z\oplus\Z/p\Z\oplus \Z/p\Z)\cap I(\Z/p\Z\oplus \Z/p\Z)=\{1, 1\},\]\[ I(\Z/p^2\Z\oplus \Z/p\Z\oplus \Z/p\Z)\cap I(\Z/p^2\Z\oplus \Z/p^2\Z\oplus\Z/p\Z)=\{2, 1\}.\]
The pair $(G(p), H(p))$ is said to satisfy $(**)$ if either $G(p)=H(p)$ or the following holds:
\[I(G(p))\backslash (I(G(p))\cap I(H(p)))=\{a_1, a_2, \dots\},\]\[ I(H(p))\backslash (I(G(p))\cap I(H(p)))=\{b_1, b_2, \dots\},\]
\[b_1>a_1>b_2>a_2>\dots. \]
\begin{example}\label{wstar}
For the pair $$(G, H)=(\Z/\Z p^5\oplus \Z/\Z p^3\oplus \Z/\Z p^2, \;\;\Z/\Z p^5\oplus \Z/\Z p^3\oplus \Z/ \Z p^3\oplus \Z/\Z p),$$
one has
\[I(G(p))\cap I(H(p))=\{5, 3\},\]\[I(G(p))\backslash (I(G(p))\cap I(H(p)))=\{2\},\;\; I(H(p))\backslash (I(G(p))\cap I(H(p)))=\{3, 1\},\]
and $(G(p), H(p))$ satisfies $(**)$.
We also have
\[I((G\oplus H)(p))=I(G(p)\oplus H(p))=\{5, 5, 3, 3, 3, 2, 1\}\]
and
\begin{align*}
 &(I(G(p))\backslash (I(G(p))\cap I(H(p))))\sqcup (I(H(p))\backslash (I(G(p))\cap I(H(p)))) \\
=&\{3, 2, 1\}\\
=&\{n\in \mathbb{N}\;|\; n\;\text{appears in}\; I((G\oplus H)(p))\;\text{odd times}\}
\end{align*} 
\end{example}
The following lemma is proved in \cite[Appendix]{Sogabe2022}.
\begin{lemma}[{\cite[Section 2.5, Proposition 2.18, 2.19]{Sogabe2022}}]\label{star}
Let $g$ (resp. $h$) be an element of $G$ (resp. $H$).
\begin{enumerate}
\renewcommand{\theenumi}{(\roman{enumi})}
\renewcommand{\labelenumi}{\textup{\theenumi}}
\item If $g$ is a torsion element in $G$, 
then $((G/\Z g)(p), G(p))$ satisfies $(**)$ for all $p$.
\item If $g$ is not a torsion element in $G$, 
then $(G(p), (G/\Z g)(p))$ satisfies $(**)$ for all $p$.
\item If $G\cong H$ and $G/\Z g\cong H/\Z h$ hold, then there is an isomorphism $G\cong H$ mapping $g$ to $h$ (i.e., $(G, g)\cong (H, h)$).
\end{enumerate}
\end{lemma}
The equality 
\begin{equation}\label{rkg}\rank(G\oplus (G/\Z g))=2\rank (G)-w(G, g) \equiv w(G, g) 
\quad (\operatorname{mod}\; 2)\end{equation} 
shows the following lemma.
\begin{lemma}\label{pari}
If  $G\oplus (G/\Z g)$ is isomorphic to $H\oplus (H/\Z h)$, one has $w(G, g)=w(H, h)$.
\end{lemma}
Applying Lemma \ref{star} and Lemma \ref{pari},
we prove the following proposition.
\begin{proposition}\label{comb}
For pairs $(G, g)$ and $(H, h)$ of finitely generated abelian groups $G, H$ and their elements $g\in G, h\in H$,
one has $G\cong H$ and $G/\Z g\cong H/\Z h$ if and only if  $G\oplus (G/\Z g)\cong H\oplus (H/\Z h)$.
\end{proposition}
\begin{proof}
We will show that there are two isomorphisms $G\cong H, \;G/\Z g\cong H/\Z h$ under the assumption $G\oplus (G/\Z g)\cong H\oplus (H/\Z h)$.

First,
we consider the case $g\in {\rm Tor}\; (G)$.
Lemma \ref{pari} implies $h\in {\rm Tor}\; (H)$
and $$\rank (G)=\frac{1}{2}\rank(G\oplus (G/\Z g))=\frac{1}{2}\rank (H\oplus (H/\Z h))=\rank(H).$$
The assumption implies $$G(p)\oplus (G/\Z g)(p)=(G\oplus (G/\Z g))(p)\cong (H\oplus (H/\Z h))(p)=H(p)\oplus (H/\Z h)(p),$$
and both of the pairs $((G/\Z g)(p), G(p))$ and $((H/\Z h)(p), H(p))$ satisfy $(**)$ by Lemma \ref{star} (1).
By the definition of $(**)$, a number $n\in \mathbb{N}$ appears in $I((G\oplus (G/\Z g))(p))$ odd times if and only if  $n$  appears in either $I(G(p))\backslash (I((G/\Z g)(p))\cap I(G(p)))$ or $I((G/\Z g)(p))\backslash (I((G/\Z g)(p))\cap I(G(p)))$.
Thus, we have
\[
\{n\in \mathbb{N}\;|\; n \;\text{appears in}\; I((G\oplus (G/\Z g))(p))\; \text{odd times}\}
=\{b_1, a_1, b_2, a_2, \dots\}\] 
where $b_i, a_i$ are the integers such that (see also Example \ref{wstar})
\[\{b_i\}=I(G(p))\backslash (I((G/\Z g)(p))\cap I(G(p))),\; \{a_i\}=I((G/\Z g)(p))\backslash (I((G/\Z g)(p))\cap I(G(p))),\]
\[b_1>a_1>b_2>a_2>\dots.\]
Since the set $$\{n\in \mathbb{N}\;|\; n \;\text{appears in}\; I((G\oplus (G/\Z g))(p))\; \text{odd times}\}$$ is uniquely divided into two sets $\{b_i\}, \{a_i\}$ so that $b_1>a_1>b_2>a_2>\dots$ holds, one has
\begin{equation}
\{b_i\}=I(H(p))\backslash (I((H/\Z h)(p))\cap I(H(p))),\label{B}\end{equation}
\begin{equation}
\{a_i\}=I((H/\Z h)(p))\backslash (I((H/\Z h)(p))\cap I(H(p))).\label{A}\end{equation}
Note that
\begin{align*}
I((G\oplus (G/\Z g))(p))=&I(G(p))\sqcup I((G/\Z g)(p))\\
=&( I((G/\Z g)(p))\cap I(G(p)))\sqcup (I((G/\Z g)(p))\cap I(G(p)))\\
&\sqcup (I(G(p))\backslash (I((G/\Z g)(p))\cap I(G(p))))\\
&\sqcup (I((G/\Z g)(p))\backslash (I((G/\Z g)(p))\cap I(G(p))))\\
=&( I((G/\Z g)(p))\cap I(G(p)))\sqcup (I((G/\Z g)(p))\cap I(G(p)))\sqcup\{b_i\}\sqcup\{a_i\}.
\end{align*}
By (\ref{B}) and (\ref{A}), the assumption $I((G\oplus (G/\Z g))(p))=I((H\oplus (H/\Z h))(p))$ implies
\begin{align*}
&( I((G/\Z g)(p))\cap I(G(p)))\sqcup (I((G/\Z g)(p))\cap I(G(p)))\\
=&(I((H/\Z h)(p))\cap I(H(p)))\sqcup (I((H/\Z h)(p))\cap I(H(p))),\end{align*}
and one has 
\[I((G/\Z g)(p))\cap I(G(p))=I((H/\Z h)(p))\cap I(H(p)), \]
\[I(G(p))=I(H(p)),\quad I((G/\Z g)(p))=I((H/\Z h)(p)).\]
This shows $(G, G/\Z g)\cong (H, H/\Z h)$.

Second,
we consider the case $g\not \in {\rm Tor}(G)$.
By Lemma \ref{pari},
one has $h\not \in {\rm Tor} (H)$ and $\rank (G)=\rank (H)$.
The same argument for $(G(p), (G/\Z g)(p))$ and $(H(p), (H/\Z h)(p))$ as in the case $g\in {\rm Tor} (G)$ shows
\[I((G/\Z g)(p))\backslash (I((G/\Z g)(p))\cap I(G(p)))=\{b_i\}=I((H/\Z h)(p))\backslash (I((H/\Z h)(p))\cap I(H(p))).\]
\[I(G(p))\backslash (I((G/\Z g)(p))\cap I(G(p)))=\{a_i\}=I(H(p))\backslash (I((H/\Z h)(p))\cap I(H(p))),\]
\[I((G/\Z g)(p))\cap I(G(p))=I((H/\Z h)(p))\cap I(H(p)).\]
Thus, one has
\[I((G/\Z g)(p))=I((H/\Z h)(p)),\quad I(G(p))=I(H(p)),\]
and we conclude $(G, G/\Z g)\cong (H, H/\Z h)$.
\end{proof}
\begin{remark}\label{rco}
For a given $G\oplus (G/\Z g)$,
one can recover $(G, G/\Z g)$ as follows.\\
\noindent{\bf Step 1} If $\rank (G\oplus (G/\Z g))$ is odd, we set $\rank (G):=\frac{1}{2}(1+\rank(G\oplus (G/\Z g)))$.
If $\rank (G\oplus (G/\Z g))$ is even, we set $\rank (G):=\frac{1}{2}\rank(G\oplus (G/\Z g))$.\\
\noindent{\bf Step 2} Determine the set $$\{n\in \mathbb{N}\;|\; n \;\text{appears in}\; I((G\oplus (G/\Z g))(p))\; \text{odd times}\}=\{b_1>a_1>b_2>a_2>\dots\}.$$
\noindent{\bf Step 3} If $\rank (G\oplus (G/\Z g))$ is odd, we set 
\[I((G/\Z g)(p))\backslash (I((G/\Z g)(p))\cap I(G(p))):=\{b_i\},\quad I(G(p))\backslash (I((G/\Z g)(p))\cap I(G(p))):=\{a_i\}.\]
If $\rank (G\oplus (G/\Z g))$ is even, we set
\[I((G/\Z g)(p))\backslash (I((G/\Z g)(p))\cap I(G(p))):=\{a_i\},\quad I(G(p))\backslash (I((G/\Z g)(p))\cap I(G(p))):=\{b_i\}.\]
\noindent{\bf Step 4} Determine $I((G/\Z g)(p))\cap I(G(p))$ from the equality
\[(I((G/\Z g)(p))\cap I(G(p)))\sqcup(I((G/\Z g)(p))\cap I(G(p)))=I((G\oplus (G/\Z g))(p))\backslash \{b_1, a_1, b_2, a_2, \dots\}.\]
These 4 steps give us  $$G(p),\; (G/\Z g)(p),\; \rank(G), \;\rank (G/\Z g):=\rank (G\oplus (G/\Z g))-\rank(G).$$
\end{remark}
We present the following main theorem in this paper.
\begin{theorem}\label{MTfg}
Let $\A, \B$ be unital Kirchberg algebras with finitely generated K-groups.
Then the followings are equivalent:
\begin{enumerate}
\renewcommand{\theenumi}{(\roman{enumi})}
\renewcommand{\labelenumi}{\textup{\theenumi}}
\item $\A\cong\B.$
\item $\Extt^1(\A)\cong \Extt^1(\B) \, \text{ and } \,  \Extw^0(\A)\cong \Extw^0(\B). \label{tw}$
\item $\Extt^i(\A)\cong \Extt^i(\B),\quad i=0, 1.\label{tt}$
\end{enumerate}
%
\end{theorem}
\begin{proof}
Since 
\begin{equation}\label{eq:5.4}
\Exts^0(\A)\oplus\Z^{(1-w(\Exts^1(\A), \iota_\A))}\cong\Extw^0(\A)
\end{equation}
 holds by (\ref{E6}), there exists an element 
$g\in G=\Extw^0(\A)$ satisfying $G/\Z g=\Exts^0(\A)$, 
 and the isomorphism $\Extt^0(\A)\cong \Extt^0(\B)$ 
 implies 
$\Exts^0(\A)\cong\Exts^0(\B)$ and $\Extw^0(\A)\cong\Extw^0(\B)$ by Proposition \ref{comb}.
Thus, the condition \ref{tt} implies \ref{tw}, 
and it is now enough to show that \ref{tw} implies $\A\cong\B$.
Since $\Extw^1(\A)\cong \Exts^1(\A)/\Z \iota_\A$,
the assertion \ref{tw} with Proposition \ref{comb} 
tells us that there exist isomorphisms 
$$
\Exts^1(\A)\cong \Exts^1(\B),\quad \Extw^1(\A)\cong \Extw^1(\B).
$$
As the two groups $\Extw^0(\A)$ and $\Exts^1(\A)$
determine $\Exts^0(\A)$
because of \eqref{eq:5.4},
the conditions
$\Exts^1(\A) \cong \Exts^1(\B)$ 
and
$\Extw^0(\A)\cong \Extw^0(\B)$
imply    
$\Exts^0(\A) \cong \Exts^0(\B).$
Hence Theorem \ref{extcp} shows $\A\cong \B$.
\end{proof}
\begin{remark}
By \eqref{eq:5.4}, 
we know that the assertion \ref{tw} above is equivalent to 
the following condition:
\begin{equation*}
\Extt^1(\A)\cong \Extt^1(\B) \, \text{ and } \, 
 \Exts^0(\A)\cong \Exts^0(\B). 
\end{equation*}
\end{remark}
By Theorem \ref{extcp}, we obtain the following corollary.
\begin{corollary}
\begin{align*}
&\{(\Extt^1(\A), \Extt^0(\A))\; |\; \A : \text{unital Kirchberg algebra with finitely generated K-groups}\}\\
=&\{(G_1\oplus (G_1/\Z g), \; G_0\oplus G_0 \oplus \Z^{1-w(G_1, g)})\; |\; G_i : \text{finitely generated abelian groups}, \;g\in G_1\}.
\end{align*}
\end{corollary}

In the case of Cuntz--Krieger algebras, 
the single invariant 
$\Extt^1(\OA) (=\Exts^1(\OA)\oplus \Extw^1(\OA))$ 
determines the isomorphism class of $\OA$ (see Corollary \ref{OAinv}).
For an $N\times N$ matrix $A$ with entries in $\{0,1\}$,
the isomorphism $\Extw^1(\OA)\cong\Z^N/(I_N-A)\Z^N$ 
is well-known (see \cite{CK}).
We define another $N\times N$ matrix $\widehat{A}$
with entries in $\{0,1\}$ by
\[\widehat{A}=A+R_1-AR_1,\quad 
R_1:=
\begin{bmatrix}
1&\cdots&1\\
0&\cdots &0\\
\vdots&\ddots&\vdots\\
0&\cdots&0
\end{bmatrix}.
\]
As \cite{MaPre2021exts} shows 
$\Exts^1(\OA)\cong\Z^N/(I_N-\widehat{A})\Z^N,$
we have 
$$
\Extt^1(\OA)\cong \Z^{2N}/(I_{2N}-\widehat{A}\oplus A)\Z^{2N}.
$$
\begin{corollary}\label{OAinv}
Let $A=[A(i,j)]_{i,j=1}^N, B=[B(i,j)]_{i,j=1}^M$ 
be irreducible non-permutation matrices with entries in $\{0,1\}$.
Then there exists an isomorphism $\OA\cong\OB$
if and only if one has
\[\Z^{2N}/(I_{2N}-\widehat{A}\oplus A)\Z^{2N}\cong\Z^{2M}/(I_{2M}-\widehat{B}\oplus B)\Z^{2M}.\]
\end{corollary}
\begin{proof}
We show the if part, and it is enough to show $\Extw^0(\OA)\cong \Extw^0(\OB)$ under the assumption 
$\Extt^1(\OA)\cong \Extt^1(\OB)$ by Theorem \ref{MTfg}.
Note that 
\begin{align*}
\Extw^0(\OA)&\cong \sqK(\OA, \mathbb{C})\\
&\cong \K_1(\OA)\\
&=\Z^{\rank(\K_0(\OA))}\\
&=\Z^{\rank(\Extw^1(\OA))}\\
&=\Z^{\rank(\Exts^1(\OA))-w(\Exts^1(\OA), \iota_{\OA})}.\end{align*}
Since $\Extw^1(\OA)\cong \Exts^1(\OA)/\Z\iota_{\OA}$, Lemma \ref{pari} and (\ref{rkg}) show
\[w(\Exts^1(\OA), \iota_{\OA})=w(\Exts^1(\OB), \iota_{\OB}),\quad 2\rank(\Exts^1(\OA))=2\rank(\Exts^1(\OB)),\]
\[\rank(\Exts^1(\OA))-w(\Exts^1(\OA), \iota_{\OA})=\rank(\Exts^1(\OB))-w(\Exts^1(\OB), \iota_{\OB}).\]
Thus, one has $\Extw^0(\OA)\cong \Extw^0(\OB)$.
\end{proof}

\medskip

The present paper is a revised version of the paper entitled
"$\K$-theoretic invariants for unital Kirchberg algebras 
with finitely generated $\K$-groups"
which appeared in arXiv:2408.09359.

\medskip

{\it Acknowledgment:}
K. Matsumoto is supported by JSPS KAKENHI Grant Number 24K06775.
T. Sogabe is supported by JSPS KAKENHI Grant Number 24K16934.

\end{document}